\documentclass[11pt,reqno]{amsart}

\usepackage{amssymb}
\usepackage{amscd}
\usepackage{amsfonts}
\usepackage{setspace}
\usepackage{version}

\numberwithin{equation}{section}
   \theoremstyle{plain}
\newtheorem{theorem}[subsection]{Theorem}

\newtheorem{lemma}[subsection]{Lemma}
\newtheorem{corollary}[subsection]{Corollary}
\newtheorem{definition}[subsection]{Definition}

\newtheorem*{mainthm3-repeat}{Theorem \ref{mainthm3}}
\newtheorem*{girth-repeat}{Lemma \ref{girth}}
\newtheorem*{jones-result-repeat}{Lemma \ref{jones-result}}
\newtheorem*{zariski-dense-repeat}{Lemma \ref{zariski-dense}}

\renewcommand{\leq}{\leqslant}
\renewcommand{\geq}{\geqslant}
\newsavebox{\proofbox}
\savebox{\proofbox}{\begin{picture}(7,7)  \put(0,0){\framebox(7,7){}}\end{picture}}

\newcommand\E{\mathbb{E}}
\newcommand\Z{\mathbb{Z}}

\newcommand\N{\mathbb{N}}

\newcommand\SL{\operatorname{SL}}

\newcommand\Cay{\operatorname{Cay}}

\newcommand\Sz{\operatorname{Sz}}
\newcommand\Sp{\operatorname{Sp}}

\newcommand\tr{\operatorname{tr}}
\renewcommand\P{\mathbb{P}}
\newcommand\F{\mathbb{F}}

\newcommand\eps{\varepsilon}
\newcommand\id{\operatorname{id}}

\begin{document}
\title[Suzuki Groups as expanders]{Suzuki Groups as expanders}

\author{Emmanuel Breuillard}
\address{Laboratoire de Math\'ematiques
Universit\'e Paris-Sud 11, 91405 Orsay cedex,
France }
\email{emmanuel.breuillard@math.u-psud.fr}
\author{Ben Green}
\address{Centre for Mathematical Sciences\\
Wilberforce Road\\
Cambridge CB3 0WA\\
England }
\email{b.j.green@dpmms.cam.ac.uk}
\author{Terence Tao}
\address{UCLA Mathematics Department,
Los Angeles, CA 90095-1555, USA}
\email{tao@math.ucla.edu}

\begin{abstract}
We show that pairs of generators for the family $\Sz(q)$ of Suzuki groups may be selected so that the corresponding Cayley graphs are expanders. By combining this with several deep works of Kassabov, Lubotzky and Nikolov, this establishes that the family of all non-abelian finite simple groups can be made into expanders in a uniform fashion.
\end{abstract}

\maketitle

\section{Introduction}

Let $X$ be a graph and let $\eps > 0$ be a real number. We say that $X$ is an $\eps$-expander if, for all sets $A$ consisting of at most half the vertices of $X$, we have $|\partial A| \geq \eps |A|$. Here, $\partial A$ refers to the boundary of $A$, that is to say those vertices that are joined to a vertex in $A$ but do not themselves lie in $A$. There is a huge literature on \emph{expander graphs}, and we refer the reader to \cite{hoory-linial-wigderson} for a comprehensive discussion. Much attention has been devoted to the particular case of \emph{Cayley graphs} $\Cay(G,S)$. Given a finite group $G$ and a symmetric generating set $S$, one defines $\Cay(G,S)$ to be the graph with vertex set $G$ in which $x$ and $y$ are joined to an edge if and only if $x = ys$ for some $s \in S$. Note that if $|S| = k$ then the Cayley graph is $k$-regular.

The following remarkable result was announced in \cite{pnas}, based on several earlier works of subsets of the authors of that paper.

\begin{theorem}[Every* non-abelian finite simple group is an expander]\label{kln-theorem}
There exist $k \in \N$ and $\eps > 0$ such that, for every* non-abelian finite simple group $G$, one may select a symmetric set $S$ of $k$ generators for which $\Cay(G,S)$ is an $\eps$-expander.
\end{theorem}

The asterisk means that this is true with the possible exception of a single family of simple groups, the \emph{Suzuki groups} $\Sz(q)$. We will recall the definition of these in Section \ref{suzuki-sec}.  We remark that the main reason why the techniques in \cite{pnas} do not extend to the Suzuki case is that those methods rely on $G$ containing an embedded copy of $\SL_2(q')$ for some large $q'$; however this is not so for Suzuki groups since $|\SL_2(q')|$ is always divisible by $3$, whereas $|\Sz(q)|$ is not.

Our main aim in this paper is to remove the lacuna in Theorem \ref{kln-theorem}.

\begin{theorem}[Suzuki groups are expanders]\label{mainthm}
There exists $\eps > 0$ such that, for every Suzuki group $G = \Sz(q)$, one may select a generating pair $a,b$ in $G$ such that, setting $S := \{a^{\pm 1}, b^{\pm 1}\}$, the Cayley graph $\Cay(G,S)$ is an $\eps$-expander. Thus Theorem \ref{kln-theorem} is true without any exception.
\end{theorem}

\emph{Remarks.} Our method is probabilistic and thus does not provide an explicit generating set. In fact, we shall prove that a random pair of elements $a,b$ will generate $G$ and have the above expansion property with probability going to $1$ as $q \rightarrow \infty$.

It is remarked in \cite{pnas} that careful computation ought to yield $k = 1000$ and $\eps = 10^{-10}$ as acceptable values in Theorem \ref{kln-theorem}. As it stands our method, though it gives the rather superior value $k = 4$, does not give any explicit value of $\eps$. However one could in principle replace all of the quantitative algebraic geometry arguments in Appendix A of \cite{bgt} by effective arguments, rather than using ultrafilters as we did there. Furthermore, this would only need to be done in the case $G = \Sp_4$ of the main theorem in \cite{bgt}.

Finally, let us note that the tendency to define group expansion in terms of graphs is a matter of custom, designed to draw attention to the much wider world of graph expanders. However, for our purposes a completely equivalent definition involving only the group is as follows: $G$ is an expander with generating set $S$ if and only if  $|A \triangle AS| \geq \eps |A|$ for all sets $A \subseteq G$ with $|A| \leq |G|/2$. Perhaps even more naturally, $G$ is an expander with generating set $S$ if and only if $|A S'| \geq (1 + \eps)|A|$ for all sets $A$ with $|A| \leq |G|/2$, where $S' := S \cup \{\id_G\}$.\vspace{11pt}

\textsc{Notation.}  We use the asymptotic notation $X = O(Y)$, $X \ll Y$, or $Y \gg X$ to denote the estimate $|X| \leq C Y$ for some absolute constant $C$.  If we need $C$ to depend on additional parameters then we will indicate this by subscripts; thus for instance $X \ll_l Y$ denotes the estimate $|X| \leq C_l Y$ for some $C_l$ depending only on $l$.  We use $o_{q \to \infty}(1)$ to denote a quantity that is bounded in magnitude by $c(q)$ for some quantity $c(q)$ depending on $q$ that goes to zero as $q \to \infty$.  If $A$ is a non-empty finite set, we use $\E_{x \in A}$ as shorthand for $\frac{1}{|A|} \sum_{x \in A}$, where $|A|$ denotes the cardinality of $A$.\vspace{11pt}

\textsc{Acknowledgments.} We would like to thank A.~Lubotzky, C.~Meiri and A.~Zuk for interesting discussions
regarding this problem, and L. Pyber for helpful comments.

EB is supported in part by the ERC starting grant 208091-GADA.
BG was, while this work was being carried out, a fellow at the Radcliffe Institute at Harvard. He is very happy to thank the Institute for proving excellent working conditions.
TT is supported by a grant from the MacArthur Foundation, by NSF grant DMS-0649473, and by the NSF Waterman award.

\section{An overview of the argument}

The basic scheme of the argument is the same as that used by Bourgain and Gamburd \cite{bg-sl2} to show that Zariski-dense (or random) pairs of generators give expanders in $\SL_2(\F_p)$. The argument requires three main ingredients.  The first is a ``Helfgott-type result'', that is to say a statement about the structure of ``approximate subgroups'' of $G$, of a type first obtained in \cite{helfgott}. The notions of approximate subgroup and of ``control'' (which features below) were also used in the context of Bourgain and Gamburd's methods in \cite{bgt}, \cite{bgt-expansion}.

\begin{theorem}\label{helf}
Suppose that $A \subseteq \Sz(q)$ is a $K$-approximate group that generates $\Sz(q)$. Then $A$ is $K^{O(1)}$-controlled by either the trivial group $\{\id\}$ or the whole group $\Sz(q)$.
\end{theorem}

\begin{proof} This follows immediately from \cite[Theorem 4]{pyber-szabo}, since $\Sz(q)$ is a simple group of Lie type.  It may also be deduced from \cite{bgt}, using the the fact that $\Sz(q)$ is a ``sufficiently Zariski-dense'' subgroup of the Chevalley group $\Sp_4(\overline{\F_2})$ of $4 \times 4$ symplectic matrices over $\overline{\F_2}$, that is to say the smallest degree of any nontrivial polynomial on $\Sp_4(\overline{\F_2})$ that vanishes on $\Sz(q)$ tends to infinity with $q$. That this is so follows from the work of Jones \cite{jones}, reproduced as Lemma \ref{zariski-dense} in the present paper; it was used in a related context by Larsen \cite{larsen}.
\end{proof}

The next ingredient is an assertion that Suzuki groups are highly quasirandom in the sense of Gowers \cite{gowers}.

\begin{theorem}[Quasirandomness]\label{quasi} The smallest dimension of an irreducible representation of $\Sz(q)$ is bounded below by $cq^{3/2}$ for some absolute constant $c > 0$.
\end{theorem}
\begin{proof} This follows from the paper of Landazuri and Seitz \cite{landazuri-seitz}.  The exponent $3/2$ is not
essential for our purposes; any lower bound which was polynomial in $q$ would suffice.  \end{proof}

The third ingredient is not in the previous literature and is therefore the beef of our paper: it asserts that random walks on the Suzuki group do not concentrate in a subgroup.

\begin{theorem}[Non-concentration estimate]\label{non-conc}  For some choice of $a,b \in \Sz(q)$ and for some constants $C$ and $\delta > 0$, we have
\begin{equation}\label{non-conc-eq} \sup_{H < G} \mu_{a,b}^{(n)}(H) < q^{-\delta}\end{equation}for all $n \geq C \log q$, where the supremum ranges over all proper subgroups $H$ of $G$, and where $\mu_{a,b}$ is the probability measure assigning weight $1/4$ to each of the four points $a,a^{-1}, b, b^{-1}$, and $\mu_{a,b}^{(n)}$ is the $n$-fold convolution.
\end{theorem}

In fact we will prove that this statement holds for a proportion $1 - o_{q \rightarrow \infty}(1)$ of all pairs $(a,b)$ in $\Sz(q) \times \Sz(q)$.

The Bourgain-Gamburd argument from \cite{bg-sl2} yields Theorem \ref{mainthm} as a consequence of Theorems \ref{helf}, \ref{quasi} and \ref{non-conc}; see \cite{green-survey} or \cite{bgt-expansion} for further discussion. Thus, the only remaining task is to establish Theorem \ref{non-conc}.

There are several steps necessary to do this, but we take advantage of the fact that maximal proper subgroups of $\Sz(q)$ are either 3-step solvable or else conjugates of $\Sz(q_0)$ for some $q_0 = q^{1/r}$, $r \in \N$, $r \geq 3$ (note that $q = 2^{2n+1}$, so we could not have $r = 2$; this point will be helpful later on); see Lemma \ref{subgroup-structure}. The solvability is particularly helpful and allows us to avoid the theory of random matrix products, important in more recent applications of the Bourgain-Gamburd method, by following a simpler argument very similar to that used in \cite{bg-sl2}. Another key ingredient will be that almost all pairs $(a,b)$ in $\Sz(q)$ satisfy no relation of length up to $\kappa \log q$, for some $\kappa>0$, a fact first proved in \cite{ghssv} (see also Appendix \ref{girth-app} below).


\section{Basic facts about Suzuki groups}\label{suzuki-sec}

There are various conceptually enlightening ways to define the Suzuki groups: we refer the reader to \cite{carter, wilson} for some of them. For our purposes it is more convenient to proceed using a quite explicit presentation of these groups using $4 \times 4$ matrices over fields of characteristic two (although we will not, in this paper, really need to know the detailed form of this presentation). Such a presentation was given in Suzuki's original paper \cite{suzuki}. A closely-related parametrisation was used in a paper of \cite{jones}. Since we will require results from Jones' paper, we use his particular parametrisation below.

Here, and for the rest of the paper, we will set $q := 2^{2n + 1}$ and consider the finite field $\F_q$. Set $\theta := 2^{n+1}$. Then, for any $x \in \F_q$, we have $(x^{\theta})^{\theta} = x^2$; that is to say, the map $x \mapsto x^{\theta}$ acts as a ``square root'' of the Frobenius map $x \mapsto x^2$.  As $\F_q$ has characteristic $2$, the map $x \mapsto x^\theta$ is of course an automorphism.

\begin{definition}[Suzuki group]\label{suz-def} Suppose that $a, b, \alpha, \beta \in \F_q$ and $c, \gamma \in \F_q^{\times}$.
Define $4 \times 4$ matrices over $\F_q$ by
\[ u(a,b, \alpha, \beta) := \begin{pmatrix} 1 & 0 & 0 & 0 \\ \alpha & 1 & 0 & 0 \\ \alpha a + \beta & a & 1 & 0 \\ \alpha^{2}a + \alpha\beta + b & \beta & \alpha & 1\end{pmatrix},\]
\[ d(c, \gamma) := \begin{pmatrix} c \gamma & 0 & 0 & 0 \\ 0 & \gamma & 0 & 0 \\ 0 & 0 & \gamma^{-1} & 0 \\ 0 & 0 & 0 & \gamma^{-1} c^{-1}\end{pmatrix}, T := \begin{pmatrix} 0 & 0 & 0 & 1 \\ 0 & 0 & 1 & 0 \\ 0 & 1 & 0 & 0 \\ 1 & 0 & 0 & 0 \end{pmatrix}.\]
Then, setting $U(\alpha,\beta):=u(\alpha^{\theta}, \beta^{\theta}, \alpha, \beta)$ and $D(\gamma):=d(\gamma^{\theta},\gamma)$, we define
\begin{align*} \Sz(q) := \{ U(\alpha,\beta) D(\gamma) T U(\alpha', \beta') & : \alpha, \alpha',\beta, \beta' \in \F_q, \gamma \in \F_q^{\times}\} \\ & \cup \{ U(\alpha, \beta) D(\gamma) : \alpha, \beta \in \F_q, \gamma \in \F_q^{\times}\}.\end{align*}
\end{definition}

Throughout the paper we write $B$ for the subgroup consisting of the products $U(\alpha, \beta)D(\gamma)$. The reader may care to check that the parametrisations given in the definition above are unique. Thus $|B| = q^2 (q-1)$ and $|\Sz(q)| = q^2 (q^2 + 1) (q-1) \sim q^5$. Moreover $\Sz(q)$ is a subgroup of $\Sp_4(q):=\{A \in \SL_4(q), A^tTA=T\}$ and the matrices of the form $u(a,b,\alpha,\beta)d(c,\gamma)$ parametrise a Borel subgroup $B_0$ of $\Sp_4(q)$.

Let us now detail several lemmas concerning these groups. First, we need the following result concerning their subgroup structure.

\begin{lemma}[Subgroups of $\Sz(q)$]\label{subgroup-structure}
Every proper subgroup $H < \Sz(q)$ is either a conjugate of $\Sz(q_0)$ for some subfield $\F_{q_0} \subsetneq \F_q$ or else is a $3$-step solvable group.
\end{lemma}
\begin{proof}
This is a consequence of \cite[Theorem 4.1]{wilson}. The notation used there hails from the \emph{Atlas of finite groups} and is not necessarily standard. Note, for example, that in parts (iii) and (iv) of \cite[Theorem 4.1]{wilson} one sees the notation $C_n : 4$, which refers to a semidirect product of the cyclic group $\Z/n\Z$ by the elementary abelian group $\Z/2\Z \times \Z/2\Z$.
\end{proof}

The next lemma asserts that $\Sz(q)$ is not contained in any low-complexity subvariety of $\Sp_4(\overline{\F_2})$; it is needed in order to give the second proof of Theorem \ref{helf} using the results in \cite{bgt}.

\begin{lemma}[$\Sz(q)$ is sufficiently Zariski dense]\label{zariski-dense}
The Suzuki group $\Sz(q)$ is not contained in any proper algebraic subgroup of $\Sp_4(\overline{\F_2})$ of complexity $M_q$, where $M_q \rightarrow \infty$ as $q \rightarrow \infty$.
\end{lemma}

\begin{proof}
See Appendix \ref{jones-app}.  In fact we can take $M_q \sim q^{1/2}$, although the exact rate is not of importance for our argument.
\end{proof}

\emph{Remark.} For a full discussion of the notion of complexity of an algebraic variety, see our earlier paper \cite{bgt}. In the context of this lemma, it simply means that the smallest degree of a polynomial which vanishes on $\Sz(q)$ but not on $\Sp_4(\overline{\F_2})$ tends to infinity as $q$ does.

We will also require some less standard facts about the Suzuki groups. The first is a result of G. Jones, giving a lower bound on the shortest word that vanishes identically on $\Sz(q)$.

\begin{lemma}[No short word laws in the Suzuki group]\label{jones-result}
Suppose that $w$ is some word in the free group $F_2$, such that $w(a,b) = \id$ for all $a,b \in \Sz(q) \setminus B$. Then $w$ has length at least $c\sqrt{q}$ for some absolute constant $c > 0$.
\end{lemma}

\begin{proof}
This follows from the proof of \cite[Lemma 5]{jones}. For the convenience of the reader we sketch the argument in Appendix \ref{jones-app}. \end{proof}

Finally, we will require the following result of Gamburd, Hoori, Shahshahani, Shalev and Vir\'ag \cite{ghssv}.

\begin{lemma}[Large girth]\label{girth}\cite{ghssv}
Let $G = \Sz(q)$. There is an absolute constant $\kappa > 0$ such that, with probability $1 - o_{q\rightarrow \infty}(1)$, a randomly chosen pair $a,b \in G$ will be such that $w(a,b) \neq \id$ for all nontrivial words in the free group $F_2$ with length at most $\kappa\log q$.
\end{lemma}

Here, recall, $o_{q \to \infty}(1)$ denotes a quantity that goes to zero as $q \to \infty$.
Most of the details in the paper \cite{ghssv} are concerned with the case of Chevalley groups, and mention of twisted groups such as the Suzuki groups under consideration here is confined to a few lines. More importantly, the treatment of these groups depends on a paper of Hrushovski \cite{hrushovski} which uses methods of mathematical logic. If we were to simply quote \cite{ghssv}, then, our main theorem would be genuinely ineffective\footnote{The ineffectivity would be of the following type: one could give an explicit expansion constant $\kappa > 0$ which works for $\Sz(q)$, $q \geq q_0$, but no effective bound on $q_0$.}. For these reasons we give a self-contained and rather elementary proof of Lemma \ref{girth} in Appendix \ref{girth-app}, leading to a value of $\kappa$ which is half that obtained in \cite{ghssv}. Although it is not strictly necessary for our paper we then indicate a more complicated, but still elementary, proof of exactly the bound in \cite{ghssv}.

\section{The nonconcentration estimate}

The aim of this section is to prove Theorem \ref{non-conc}.
We may assume that $q$ is sufficiently large, since the claim is trivial for bounded $q$ by selecting $a, b$ to be an arbitrary pair of generators of $\Sz(q)$.  (It has long been known that such generators exist: see \cite{steinberg}.)

Note that it will be enough to prove this estimate for $n$ of the order of $\log q$ (up to multiplicative constants). Indeed, if $\mu_{a,b}^{(2n)}(H) < q^{-\delta}$ for some such $n$, then $\mu_{a,b}^{(n+m)}(H) = \E_x \mu_{a,b}^{(m)}(x)\mu_{a,b}^{(n)}(x^{-1}H) \leq (\mu_{a,b}^{(2n)}(H))^{1/2} < q^{-\frac{\delta}{2}}$ for all $m \geq 0$, since $\mu_{a,b}^{(2n)}(H)\geq \mu_{a,b}^{(n)}(x^{-1}H)\mu_{a,b}^{(n)}(Hx)=\mu_{a,b}^{(n)}(x^{-1}H)^2$.

We will in fact show that $(\ref{non-conc-eq})$ holds for randomly selected $a$ and $b$, with probability $1 - o_{q \rightarrow \infty}(1)$ as $q \rightarrow \infty$. Note that this implies that randomly selected $a$ and $b$ do generate $\Sz(q)$ if $q$ is sufficiently large.
By Proposition \ref{subgroup-structure} this task divides into two subtasks: we must handle the \emph{algebraic} case in which $H$ is 3-step solvable and the \emph{arithmetic} case in which $H = x^{-1}\Sz(q_0)x$ for some $\F_{q_0} \subsetneq \F_q, x\in G$.

\bigskip

\emph{The algebraic case.} Our argument here is almost identical to that in \cite{bg-sl2}.

Let $F_k$ be the free group on $k$ letters, and write $W_k(L)$ for the ball of radius $L$ about $\id$ in the word metric on $F_k$. If $G$ is a group, then for each $i = 0,1,2,3,\dots$ we define the $i$-fold commutator maps $\psi_i : G^{2^{i}} \rightarrow G$ by
\[ \psi_0(g) = g,\]
\[ \psi_1(g_0,g_1) = [g_0,g_1],\]
\[ \psi_2(g_{00}, g_{01}, g_{10}, g_{11}) = [ [g_{00}, g_{01}], [g_{10}, g_{11}]],\]
\[ \psi_3((g_{\omega})_{\omega \in \{0,1\}^3}) = [ [[g_{000}, g_{001}], [g_{010}, g_{011}]  ], [ [g_{100}, g_{101}], [g_{110}, g_{111}]]   ] \] and so on.
For definiteness (though it scarcely matters) we use the group theorists' definition of commutator, namely $[x,y] = x^{-1} y^{-1} x y$.

The following lemma is a very straightforward modification of \cite[Proposition 8]{bg-sl2}.

\begin{lemma}\label{solv-free}
Suppose that $S \subseteq W_k(L)$ is a set with the property that \[ \psi_l((s_{\omega})_{\omega \in \{0,1\}^l}) = \id\] for all $2^l$-tuples $(s_{\omega})_{\omega \in \{0,1\}^l} \subseteq S$. Then $|S| \ll_l L^{2l}$.
\end{lemma}

We remark that we will only need this lemma in the case $k=2, l=3$, though it is scarcely more difficult to establish the general case.

\begin{proof}
Write $f(k,l,L)$ for the smallest function which works in the claimed bound; thus our desire is to show that $f(k,l,L) \ll_l L^{2l}$. Suppose that $S$ is a set with the stated property. Then the set $S' := \{[s,s'] : s, s' \in S\}$ is a subset of $W_k(4L)$ with the property that all the $(l-1)$-fold commutator maps are trivial on $S'$. Therefore we have
\[ |S'| \leq f(k,l-1,4L).\]
By the pigeonhole principle it follows that there exists some $a \in S'$ and $b \in S$ such that there is a set $S_0 \subseteq S$,
\begin{equation}\label{eq11}|S_0|  \geq |S|/f(k,l-1,4L),\end{equation} such that $[b,s] = a$ for all $s \in S_0$.

Suppose that $s_1,s_2 \in S_0$. Then we have $s_1^{-1} b s_1 = s_2^{-1} b s_2$, which implies that $b$ commutes with $s_1 s_2^{-1}$. By standard facts about the free group this implies that there is some word $x$ and some integers $m,n$, $|m| , |n| \leq  2L$ such that $b = x^m$ and $s_1 s_2^{-1} = x^n$. Since knowledge of $x^m$ and $m$ uniquely determines $x$ (another standard fact about the free group), we see that, assuming $b$ is fixed, there are no more than $(4L + 1)^2 < (5L)^2$ possible values for $s_1 s_2^{-1}$.

Double-counting pairs of elements of $S_0$, it follows that $|S_0|^2 \leq (5L)^2|S_0|$.
Comparing this with \eqref{eq11} of course yields
\[ |S| \leq (5L)^2 f(k,l-1,4L).\] By the definition of $f$ this means that
\[ f(k,l,L) \leq (5L)^2 f(k, l-1,4L),\] and hence
\[ f(k,l,L) \leq 5^{2l} L^2(4L)^2 \dots (4^{l-1}L)^2 \cdot f(k,0, 4^l L) \leq 5^{2l}4^{l^2} L^{2l} f(k,0,4^l L).\]
This concludes the proof.
\end{proof}

We now have enough tools to conclude the analysis of the algebraic case. Select a random pair of elements $a,b \in \Sz(q)$. By Proposition \ref{girth} these are free up to length $ \kappa \log q$ with probability $1 - o_{q \rightarrow \infty}(1)$. Set $n_0 := \kappa \log q/100$ (say), and suppose that
\[ \mu_{a,b}^{(n)}(H) \geq \eta,\] for some $n \geq n_0$ and for some $3$-step solvable group $H$. Then by a simple averaging argument there is some $x$ such that
\[ \mu_{a,b}^{(n_0)}(Hx) \geq \eta.\] As in the argument at the start of the section, this implies that
\[ \mu_{a,b}^{(2n_0)}(H) \geq \eta^2.\]
Note that $\mu_{a,b}$ is counting walks of length $2n_0$ rather than distinct words; for example if $n_0 = 3$ then $ab$ could be counted as $abb^{-1} a a^{-1} b$. However, just as in \cite[p. 637]{bg-sl2}, it follows from Kesten's celebrated thesis \cite{kesten} that no more than $(2\sqrt{3})^{2n_0}$ of the $4^{2n_0}$ walks of length $2n_0$ starting from the identity end up at any given point $x$.

Now words in $a,b$ of length up to $16(2n_0)$ behave freely, and so we can bound $\mu_{a,b}^{(2n_0)}(H)$ above by $(\sqrt{3}/2)^{2n_0}$ times $|W_2(2n_0) \cap H|$. Since $H$ is $3$-step solvable, the $3$-fold iterated commutator $\psi_3$ vanishes on $H^8$. Applying Lemma \ref{solv-free}, we thus obtain
\[ \mu_{a,b}^{(2n_0)}(H) \ll n_0^8 (\sqrt{3}/2)^{2n_0}.\]
Since $n_0 \gg \log q$, it follows that $\eta < q^{-\delta}$ for some suitably small absolute constant $\delta$. This concludes the proof of the algebraic case.

\bigskip

\emph{The arithmetic case.} Recall that our aim is to show that
\[ \sup_{x\in G, \F_{q_0} \subsetneq \F_q}\mu_{a,b}^{(n)}(x^{-1} \Sz(q_0)x) < q^{-\delta}\] for all $n \sim C\log q$, for some choice of $a,b \in G=\Sz(q)$ and for some $C, \delta >0$.

We will in fact establish the bound
\begin{equation}\label{to-prove}
\E_{a,b \in \Sz(q)}\sup_{x \in G, \F_{q_0} \subsetneq \F_q}\mu_{a,b}^{(n)}(x^{-1} \Sz(q_0)x) < q^{-\delta},\end{equation}
from which the stated result is immediate by Markov's inequality.

If $t \in \SL_4(q)$ is a matrix, write
\[ \lambda^4 + c_1(t)\lambda^3 + c_2(t)\lambda^2 + c_3(t)\lambda + 1 := \det(t+\lambda)\] for the characteristic polynomial of $t$; thus, for example, $c_1(t) = \tr(t)$ and $c_3(t) = \tr(t^{-1})$. Since the characteristic polynomial is invariant under conjugation, for given $a,b$ the supremum in \eqref{to-prove} is bounded above by

\[ \P_w\big(c_1(w(a,b)), c_2(w(a,b)), c_3(w(a,b)) \in \bigcup_{\F_{q_0} \subsetneq \F_q} \F_{q_0}\big),\] where the probability is taken over all words $w$ of length $n \sim C\log q$.  We conclude that
\begin{equation}\label{eq956}
 \eqref{to-prove} \leq \sigma_1 + \sigma_2
\end{equation}
where
\begin{equation}\label{sigma1-def}
 \sigma_1 := \E_{a,b \in \Sz(q)}\sum_{i=1}^3\sum_{x \in \bigcup_{\F_{q_0} \subsetneq \F_q} \F_{q_0}, x \neq 0} \P_w(c_i(w(a,b)) = x)
\end{equation}
and
\begin{equation}\label{sigma2-def}
 \sigma_2 := \E_{a,b \in \Sz(q)}\P_w( c_1(w(a,b)) = c_2(w(a,b)) = c_3(w(a,b)) = 0).
\end{equation}

To bound these quantities $\sigma_1$ and $\sigma_2$, we use the following simple lemma.

\begin{lemma}[Schwartz-Zippel type lemma]\label{simple-twist-bound}
Suppose that $\F=\F_q$ for some $q = 2^{2n + 1}$, $n > 1$, and that $P : \F^k \times \F^k \rightarrow \F$ is a polynomial in $2k$ variables of degree at most $d$ in any one of them. Set $\theta := 2^{n+1}$. Then either $P(x_1,\dots,x_k, x_1^{\theta},\dots, x^{\theta}_k) = 0$ for all $x_1,\ldots,x_k \in \F$, or else the probability that $P(x_1,\dots,x_k ,x^{\theta}_1,\dots,x^{\theta}_k) = 0$ for a random choice of $x_1,\dots, x_k \in \F$ is $O(kdq^{-1/2})$.
\end{lemma}

\begin{proof}
Write $f(q,k,d)$ for the least quantity that will work as a bound in this lemma. Let $P: \F^k \times \F^k \to \F$ be a polynomial such that $P(x_1,\dots, x_k, x^{\theta}_1,$ $\dots, x^{\theta}_k)$ does not vanish identically.  Write
\begin{equation}\label{eq47} P(x_1,\dots, x_k, x^{\theta}_1,\dots, x^{\theta}_k) = \sum_{0 \leq j, j' \leq d} p_{j,j'} (x_2,\dots, x_k, x^{\theta}_2,\dots, x^{\theta}_k) x_1^j x^{j'\theta}_1,\end{equation}
where the $p_{j,j'}$ are polynomials of degree at most $d$ in each variable. Then there exists some $j, j'$ for which
\[ p_{j,j'}(x_2,\dots,x_k, x^\theta_2,\dots ,x^\theta_k)\] does not vanish identically. For a randomly selected $2k$-tuple $(x_1,\dots, x_k, x^\theta_1,$ $\dots, x^\theta_k)$ we distinguish two cases: either \[ p_{j,j'}(x_2,\dots,x_k, x^\theta_2,\dots ,x^\theta_k) = 0,\] or this is not so. The chance of the first case occurring is at most $f(q,k-1,d)$. In the second case, let us fix $x_2,\dots, x_k, x^\theta_2,\dots, x^\theta_k$ and count the number of possibilities for $x_1, x^\theta_1$. In this case \eqref{eq47} becomes a nontrivial polynomial equation of the form $p(x_1,x^\theta_1) = 0$. Recalling that $x^\theta_1 = x_1^{2^{n+1}}$, we may regard this as a polynomial equation of degree at most $d(2^{n+1} + 1)$ in $x_1$, and therefore it has at most $d(2^{n+1} + 1)$ solutions. Putting these observations together we obtain the inequality
\[ f(q,k,d) \leq f(q,k-1,d) + q^{-1}d(2^{n+1} + 1).\] Iterating we obtain
\[ f(q,k,d) \leq kdq^{-1}(2^{n+1} + 1),\] which implies the claimed bound.
\end{proof}

\emph{Remark.} By using Lemma \ref{harder-twist}, which gives a bound of $2d^2$ instead of the trivial bound of $d(2^{n+1} + 1)$ for the number of solutions to $p(x_1,x_1^{\theta}) = 0$, it is possible to obtain the stronger bound $f(q,k,d) \leq 2kd^2/q$ in this lemma. We will not require this improvement in this paper.

Using Lemma \ref{simple-twist-bound} we can bound the probability that $c_i$ concentrates at a non-zero value.

\begin{lemma}[Non-concentration away from zero]\label{word-trace-lem}
Suppose that $w$ is a word in the free group $F_2$ of length $L$, that $x \in \F_q \setminus \{0\}$, and that $i = 1,2$ or $3$. Then if $a,b$ are selected randomly from $\Sz(q)$ we have
\[ \P_{a,b} ( c_i(w(a,b)) = x) \ll q^{-1/2}L.\]
\end{lemma}

\begin{proof} The reader may wish to recall the explicit definition of $\Sz(q)$ given in Definition \ref{suz-def}. The probability that either $a$ or $b$ lies in the subgroup $B$ is at most $O(q^{-2})$, which is acceptable. For all other $a,b$ we may use the parametrisation $U(\alpha, \beta) D(\gamma) T U(\alpha', \beta')$. To parametrise $a$ and $b$ we require 10 variables, which we denote by $x_1,\dots, x_{10}$. Clearing denominators arising from the appearance of $\gamma^{-1}$ and $\gamma^{-\theta}$ in $D(\gamma)$, we may rewrite the equation $c_i(w(a,b)) = x$ in the form
\[ P(x_1,\dots, x_{10}, x^{\theta}_1,\dots, x^{\theta}_{10}) = 0,\] where $P$ is a polynomial in 20 variables, of degree $O(L)$ in each of them (note that every matrix in $\Sp_4(q)$ has determinant one, so there is no issue when $w$ contains inverses of $a$ and $b$). This polynomial does not vanish identically, since when $a = b = T$ the characteristic polynomial of any word $w(a,b)$ is simply $\lambda^4 + 1$, so in this case $c_1(w(a,b)) = c_2(w(a,b)) = c_3(w(a,b)) = 0$. We are assuming, however, that $x \neq 0$.
The result now follows from Lemma \ref{simple-twist-bound}.
\end{proof}

An immediate corollary of Lemma \ref{word-trace-lem} is a bound for the quantity $\sigma_1$ from \eqref{sigma1-def}.

\begin{corollary}\label{sigma_1-cor}
We have $\sigma_1 \ll q^{-1/6} \log q$.
\end{corollary}

\begin{proof} By Lemma \ref{word-trace-lem} and the definition of $\sigma_1$ we have
\[ \sigma_1 \ll  | \bigcup_{\F_{q_0} \subsetneq \F_q} \F_{q_0}|  \cdot q^{-1/2} L.\]
Now simply observe that
\[ \sum_{\F_{q_0} \subsetneq \F_q} |\F_{q_0}| \leq \sum_{r \leq (2n+1)/3} 2^r \ll q^{1/3},\] this being a consequence of the fact that $2n + 1$ is odd, and hence has no proper factor larger than $(2n + 1)/3$.
\end{proof}

It remains to bound the quantity $\sigma_2$ from \eqref{sigma2-def}.  Note that if $c_1(w(a,b)) = c_2(w(a,b)) = c_3(w(a,b)) = 0$ then the characteristic polynomial of $w(a,b)$ is simply $\lambda^4 + 1$ and hence, by the Cayley-Hamilton theorem, $w(a,b)^4 = \id$.
Once again, the probability that either $a$ or $b$ lies in the subgroup $B \subseteq \Sz(q)$ is just $O(q^{-2})$.
It therefore suffices to bound
\[ \E_{a,b \in \Sz(q)} \P_w (a,b \notin B, w(a,b)^4 = \id).\]
The average over $w$ involves words of length only $O(\log q)$. By Lemma \ref{jones-result}, none of these words $w(a,b)^4$ (except the trivial word) vanishes identically on $\Sz(q) \setminus B$.

Parametrise pairs of elements of $a,b \in \Sz(q) \setminus B$ as before, using ten parameters $x_1,\dots, x_{10}$. The condition that $w(a,b)^4 = \id$ leads to sixteen polynomial equations, each of degree $O(\log q)$ in every variable. By Jones' result (and for $Q$ sufficiently large), at least one of these polynomials must be nontrivial. Applying
Lemma \ref{simple-twist-bound} we may conclude that
\begin{equation} \sigma_2 \ll q^{-1/2} \log q.\label{sigma_2}\end{equation}

Combining this with Corollary \ref{sigma_1-cor} and recalling the definition of $\sigma_1,\sigma_2$  (cf. \eqref{eq956}) we obtain \eqref{to-prove} as desired.
This concludes our analysis of the arithmetic case and hence the proof of Theorem \ref{non-conc}.

\section{Remarks on $\SL_2(q)$.}

All of our arguments go through (and are considerably simpler) for $\SL_2(q)$, $q=p^n$, $p$ prime, $n\geq 1$. It is, however, already known due to work of Lubotzky \cite{lubotzky} (referenced in \cite{pnas}) that the groups $\SL_2(q)$ are uniformly expanders (with varying $p$ and $n$) for some rather explicit sets of three generators and their inverses.  The methods of this paper, then, provide an alternate proof of Lubotzky's result. There are perhaps three good reasons for wishing to have such an alternative proof. Firstly, we avoid the use of any deep number-theoretic information. Secondly, we can get two generators rather than three. Thirdly, the result itself is of paramount importance in the proof of the main theorem of Kassabov-Lubotzky-Nikolov \cite{pnas}.

Everything goes through in much the same way as for the Suzuki groups. The ``Helfgott-type" result for $\SL_2(q)$ (analogous to Theorem \ref{helf}) was obtained by Helfgott himself \cite{helfgott} for $q$ prime, and was subsequently generalized to arbitrary $q$ in Oren Dinai's thesis (see \cite{dinai}). As for the proof of the non-concentration estimate, the algebraic case involves, apart from some subgroups of bounded order ($\leq 60$), only 2-step solvable groups rather than 3-step. We refer the reader to Dickson's book \cite[XII, 260]{dickson} for a careful description of the subgroup structure of $\SL_2(q)$. Finally, there is a much more elementary argument for the arithmetic case, which we now give.

Recall that the aim is to show that, for a proportion $1 - o_{q \rightarrow \infty}(1)$ of pairs $a,b \in G$, we have
\begin{equation}\label{non-conc2} \mu_{a,b}^{(n)}(x^{-1} \SL_2(q_0) x) < q^{-\delta}\end{equation} for some $\delta > 0$ and for all $n \geq C \log q$, uniformly for all $x \in G$ and all $q_0 = q^{1/r}$, $r > 1$. In $\SL_2$, the characteristic polynomial is determined by the trace, so we can proceed exactly as in the last section with the trace in place of the $c_i$'s, except now there is a somewhat easier way to obtain a bound
\begin{equation}\label{trace-bound} \P_{a,b \in \SL_2(q)}( \tr(w(a,b)) = x) \ll q^{-1+ \eps},\end{equation} which suffices for our purposes in exactly the same way that Corollary \ref{sigma_1-cor} and the bound $(\ref{sigma_2})$ did in the previous section.
Indeed, suppose that
\[ a = \begin{pmatrix} t_1 & t_2 \\ t_3 & \frac{t_2 t_3 + 1}{t_1}   \end{pmatrix}, b = \begin{pmatrix} u_1 & u_2 \\ u_3 & \frac{u_2 u_3 + 1}{u_1}   \end{pmatrix} \] with $t_1, u_1 \neq 0$. For a fixed word $w$ of length $O(\log q)$, the condition $\tr(w(a,b)) = x$ can be written (by clearing denominators $t_1, u_1$) as a polynomial of degree $O(\log q)$ in the variables $t_i, u_i$. By Borel's theorem \cite{borel} (see also \cite{larsen}), the word map $w : \SL_2 \times \SL_2 \rightarrow \SL_2$ is a dominant map as soon as $w$ is non-trivial. This implies that our polynomial is not constant. Thus it is not hard to see that the number of solutions to it is $O(q^5 \log q)$ by a somewhat simpler, untwisted, version of the counting argument given in Proposition \ref{simple-twist-bound} above. The number of pairs $a,b$ in which the top left entry of either $a$ or $b$ is zero is clearly $O(q^5)$, and this completes the proof of $(\ref{trace-bound})$.

To get $(\ref{non-conc2})$ it suffices to observe that the subset of elements in $\F_q$ which lie in a proper subfield has size $O(\sqrt{q})$. This completes the argument.

\appendix

\section{A result of G.~A.~Jones}

\label{jones-app}

In this appendix, we prove Lemmas \ref{zariski-dense} and \ref{jones-result} following G. Jones' paper \cite{jones}.

We keep the notation of Section \ref{suzuki-sec}. Let $q=2^{2n+1}$ and let $k = \overline{\F_2}$ be the algebraic closure of $\F_2$. We will view $\Sp_4(k)$ as a closed affine algebraic subset of $4 \times 4$ matrices over $k$, i.e. as a subset of $k^{16}$. We view it as an algebraic group endowed with the Zariski topology. Let $f$ be a polynomial in 16 variables over $k$. We begin by establishing the following lemma, which forms the heart of Jones' argument.

\begin{lemma}\label{jones-lemma} There is a constant $c>0$ such that if $f$ has degree at most $c \sqrt{q}$, and $f$ vanishes identically on $\Sz(q)\setminus B$, then $f$ vanishes identically on $\Sp_4(k)$.
\end{lemma}

\begin{proof} Suppose $f$ has degree at most $M:=q^{1/2}/10$ is each of the 16 variables. Recall the parametrisation of the (lower triangular) Borel subgroup $B_0$ of $\Sp_4(k)$ as $u(a,b,\alpha,\beta)d(c,\gamma)$. The set $B_0TB_0$ can be parametrised by matrices of the form $u(a,b,\alpha,\beta)d(c,\gamma)Tu(a',b',\alpha',\beta')$. It is the so-called ``big-cell" of the Bruhat decomposition of $\Sp_4(k)$. It is Zariski-open and hence Zariski-dense in $\Sp_4(k)$. Hence if $f$ vanishes entirely on $B_0TB_0$, it must vanish entirely on $\Sp_4(k)$. After multiplying $f$ by $(\gamma c)^M$ in order to clear denominators, we obtain a polynomial $P=P(a,b,\alpha,\beta, c,\gamma, a', b', \alpha', \beta')$ of degree at most $2M$ in each of the ten variables, with the property that $P = 0$ if and only if $\gamma c=0$ or if $f = 0$. It particular if $P$ vanishes on $k^{10}$, i.e. is the zero polynomial, then $f$ vanishes identically on $\Sp_4(k)$.

Recall from Definition \ref{suz-def} that any element $g \in \Sz(q)\setminus B$ can parametrised as \[ g=u(\alpha^{\theta}, \beta^{\theta}, \alpha, \beta)d(\gamma^{\theta},\gamma) T u(\alpha'^{\theta},\beta'^{\theta},\alpha',\beta').\] Let $Q$ be the polynomial in 5 variables defined as \[Q(\alpha,\beta, \gamma, \alpha', \beta')=P(\alpha^{\theta}, \beta^{\theta}, \alpha, \beta, \gamma^{\theta},\gamma, \alpha'^{\theta},\beta'^{\theta},\alpha',\beta').\]  Since (by assumption) $f$ vanishes on $\Sz(q) \setminus B$, $Q$ must vanish identically for all values of $\alpha, \beta, \gamma, \alpha', \beta'$ taken in $\F_q$. We shall prove that, due to the upper bound on $M$, this forces the polynomial $P(a,b,\alpha,\beta, c,$ $\gamma, a', b', \alpha', \beta')$ to be the zero polynomial and that will complete the argument.

Note that $Q$ has degree at most $2M(\theta+1)$ in each of its 5 variables. Since $10M(\theta+1) < q$, $Q$ must take a non-zero value on $\F_q^5$ unless $Q$ is formally zero. This follows from the well-known fact that any non-zero polynomial in $k$ variables over $\F_q$ of degree at most $d$ in each variable has at most $kdq^{k-1}$ solutions in $\F_q^k$, by the Schwartz-Zippel lemma (which one can prove by an induction on dimension argument similar to that used to prove Lemma \ref{simple-twist-bound}).  Therefore $Q$ is formally zero.

Suppose now that $P$ is formally non-trivial. Then there must be two distinct monomials $m_1=a^{m^1_a}b^{m^1_b}\alpha^{m^1_\alpha} \ldots c'^{m^1_{c'}}$ and $m_2=a^{m^2_a}b^{m^2_b}\alpha^{m^2_\alpha} \ldots  c'^{m^2_{c'}}$ appearing in $P$ which coincide after the substitution $a \mapsto \alpha^{\theta}$, $b \mapsto \beta^{\theta}$, $c \mapsto \gamma^{\theta}$. This means that $\theta m^1_a + m^1_{\alpha} = \theta m^2_a + m^2_{\alpha}$ and similarly for $b$ and $c$. However, since $2M<\theta $, this forces $m^1_a=m^2_a$ and $m^1_{\alpha}=m^2_{\alpha}$ and similarly for $b, c, \beta$ and $\gamma$, that is to say $m_1=m_2$. This contradiction implies that $P$ is formally trivial and this ends the proof of the lemma.\end{proof}

We are now in a position to deduce Lemma \ref{zariski-dense}, whose statement we recall now.

\begin{zariski-dense-repeat}
The Suzuki group $\Sz(q)$ is not contained in any proper algebraic subgroup of $\Sp_4(k)$ of complexity $M_q$, where $M_q \rightarrow \infty$ as $q \rightarrow \infty$.
\end{zariski-dense-repeat}

\begin{proof} Let $V$ be a proper closed subvariety of $\Sp_4(k)$ of complexity at most $M$ containing $\Sz(q)$. Then by the definition of complexity (cf. \cite[Section 3]{bgt}), there is a polynomial $f$ on 16 variables over $k$ and of degree at most $M$, which does not vanish identically on $\Sp_4(k)$ yet is identically zero on $V$. We are thus in a position to apply the above lemma an conclude that $M\geq c q^{1/2}$.
\end{proof}

Our other business is to establish Lemma \ref{jones-result}.

\begin{jones-result-repeat}
Suppose that $w$ is some word in the free group $F_2$, and that $w(a,b) = \id$ identically on $\Sz(q) \setminus B$. Then $w$ has length at least $c\sqrt{q}$ for some absolute constant $c > 0$.
\end{jones-result-repeat}
\begin{proof} Suppose not. The relation $w(a,b)=\id$ can be written as 16 polynomial equations of degree $\ll q^{1/2}$ in the 16 matrix coordinates. Applying Lemma \ref{jones-lemma} to each of them, we conclude that $w(a,b)=\id$ for every $a$ and $b$ in $\Sp_4(k)$. We then observe that $\Sp_4(k)$ contains a closed algebraic subgroup isomorphic to $\SL_2(k)$ (for example the subgroup of $\Sp_4$ fixing the vectors $e_1$ and $e_4$). On the other hand it follows from Borel's result\footnote{This particular fact may be also be proven more elementarily by a ping-pong argument, noting that $w$ would also have to vanish on $\SL_2(k(t)) \times \SL_2(k(t))$, where $t$ is an indeterminate. This avoids an appeal to the Tits Alternative.} \cite{borel} that $w(a,b) - \id$ cannot vanish identically on $\SL_2(k) \times \SL_2(k)$ unless $w$ is trivial. \end{proof}

\section{Random Cayley graphs of $\Sz(q)$ have large girth} \label{girth-app}

Our aim in this appendix is to supply a self-contained proof of Lemma \ref{girth}.
The argument is basically the same as that in \cite{ghssv}, only we use elementary estimates instead of the deep model-theoretic work of Hrushovski. We first recall the statement of the lemma.

\begin{girth-repeat}
Let $G = \Sz(q)$. There is an absolute constant $\kappa > 0$ such that, with probability $1 - o_{q\rightarrow \infty}(1)$, a randomly chosen pair $a,b \in G$ will be such that $w(a,b) \neq id$ for all nontrivial words $w$ in the free group $F_2$ with length at most $\kappa\log q$.
\end{girth-repeat}
\begin{proof}
Let $\kappa > 0$ be a quantity to be specified later. The number of nontrivial words $w \in F_2$ of length at most $\kappa\log q$ is  $\ll q^{\kappa \log 3}$. For each of them, let us estimate the probability that a random pair of elements $a,b \in \Sz(q)$ satisfies $w(a,b) = \id$. The probability that either $a$ or $b$ lies in the subgroup $B$ (cf. Definition \ref{suz-def}) is $O(q^{-2})$ and will be ignored. The other elements may be parametrised as $U(\alpha, \beta)D(\gamma) T U(\alpha',\beta')$. Parametrising $a$ and $b$ by ten variables $x_1,\dots,x_{10}$, the equation $w(a,b) = \id$ is equivalent to sixteen polynomial equations
\[ P_{ij}(x_1,\dots, x_{10}, x^{\theta}_1,\dots, x^{\theta}_{10}) = 0,\] one for each matrix entry, where each $P_{ij}$ has degree $O(\log q)$ in each variable (we have cleared the denominators $\gamma^{-1}$ and $\gamma^{-\theta-1}$ appearing in the expression by multiplying by a $O(\log q)$ power of $\gamma$ and $\gamma^{\theta}$ to get $P_{ij}$).

By Lemma \ref{jones-result} we do not have $w(a,b) = \id$ identically for $a,b \in \Sz(q) \setminus B$, and so at least one of these polynomial equations is nontrivial.  By Lemma \ref{simple-twist-bound}, it follows that
\[ \P_{a,b}(a,b \notin B, w(a,b) = \id) \ll q^{-1/2} \log q.\]
Summing over $w$, one sees that the probability that a randomly selected pair $a,b$ will satisfy \emph{any} word of length $\kappa \log q$ is bounded by $O(q^{\kappa \log 3 - 1/2}\log q )$. Choosing $\kappa < 1/2 \log 3$, the result follows immediately.
\end{proof}

As we remarked in the overview, our lower bound for the girth is precisely half that of \cite{ghssv}. The reason for this is the rather crude bound on the number of solutions to $p(x,x^{\theta}) = 0$ that we employed during the proof of Lemma \ref{simple-twist-bound}. By employing the following lemma instead (with appropriate modifications of Lemma \ref{simple-twist-bound}) we may recover the bound of \cite{ghssv}. This lemma is plausibly of independent interest.

\begin{lemma}\label{harder-twist}
Let $q = 2^{2n + 1}$ and write $\theta := 2^{n+1}$. Let $p(x,y) \in \F_q[x,y]$ be a nontrivial polynomial with degree at most $d$ in each of its variables. Then the number of solutions to $p(x, x^{\theta}) = 0$ with $x \in \F_q$ is at most $2d^2$.
\end{lemma}
\begin{proof} Suppose that $p(x,x^{\theta}) = 0$. Raising to the power $\theta$ and recalling that $x \mapsto x^\theta$ is an automorphism with $(x^{\theta})^{\theta} = x^2$, we obtain
\begin{equation}\label{c1} p^{\theta}(x^{\theta}, x^2) = 0.  \end{equation} Here, $p^{\theta}$ means the polynomial obtained from $p$ by raising each coefficient to the power $\theta$. Motivated by this observation let us consider the more general problem of bounding above the number of solutions to
\[ p^{\theta}(y, x^2) =  p(x,y) = 0\] with $x,y \in \F_q$.
Write $P(x,y) = p^{\theta}(y,x^2)$; note that the total degree of $P$ is at most $2d$. Let $f(x,y)$ be the highest common factor of $p$ and $P$ in $\overline{\F_q}[x,y]$. Multiplying by a suitable unit, we may take $f(x,y)$ to lie in $\F_q[x,y]$: indeed any $\mbox{Gal}(\overline{\F_q}/\F_q)$-conjugate of $f(x,y)$ is also a common factor of $p$ and $P$ and so every irreducible factor of $f(x,y)$ comes together with all of its conjugates, and the product of those is defined over $\F_q$. Write $p = fp_1$ and $P = f P_1$.

Write $d' $ for the total degree of $f$, and suppose that $d' < d$. By Bezout's theorem, the number of solutions to $p_1(x,y) = P_1(x,y) = 0$ is at most $(2d - d')(d - d')$.  All other solutions to $p(x,x^{\theta}) = 0$ must also satisfy $f(x, x^{\theta}) = 0$, and so we may thus proceed inductively to conclude that the total number of $x \in \F_q$ with $p(x,x^{\theta}) = 0$ is at most
\[ (2d - d')(d - d') + 2d^{\prime 2} \leq 2d^2,\] as required.

If $d' = d$ then we must proceed differently. In this case $p(x,y)$ divides $P(x,y)$, and so we may write
\[ P(x,y) = p^{\theta}(y, x^2) \equiv g(x,y) p(x,y)\] for some polynomial $g(x,y) \in \F_q[x,y]$. Here, and henceforth, we use $\equiv$ to denote equivalence of polynomials (and not just expressions that are equal for all substitutions of variables from $\F_q$). Making an obvious substitution, we have
\[ p^{\theta}(x^2, y^2) \equiv g(y,x^2) p(y,x^2).\]
Raising both sides to the power $\theta$ then yields
\[ p^2(x^{2\theta}, y^{2\theta}) \equiv g^{\theta}(y^{\theta}, x^{2\theta}) p^{\theta}(y^{\theta}, x^{2\theta})\] and hence that
\[ p^2(x^2, y^2) \equiv g^{\theta}(y,x^2) p^{\theta}(y, x^2).\]
 It follows that
\[ p(x,y)^2 \equiv p^2(x^2, y^2) \equiv g^{\theta}(y,x^2) p^{\theta}(y, x^2) \equiv g^{\theta}(y,x^2) g(x,y) p(x,y),\] and so
\[ p(x,y) \equiv g^{\theta}(y,x^2) g(x,y).\]
In particular $p$ is reducible over $\F_q$ and we may proceed by induction on the total degree of $p$.
\end{proof}

The second author would like to thank Michael Larsen for helpful conversations in connection with this lemma.


\begin{thebibliography}{99}

\bibitem{borel} A.~Borel, \emph{On free subgroups of semisimple groups,} Enseign. Math. (2) \textbf{29} (1983), no. 1-2, 151--164.
\bibitem{bg-sl2} J.~Bourgain and A.~Gamburd, \emph{Uniform expansion bounds for Cayley graphs of $\SL_2(\F_p)$,} Ann. of Math. (2) \textbf{167} (2008), no. 2, 625--642.
\bibitem{bgt} E.~Breuillard, B.~J.~Green and T.~C.~Tao, \emph{Approximate subgroups of linear groups,} preprint.


\bibitem{bgt-expansion}
E.~Breuillard, B.~J.~Green and T.~C.~Tao, \emph{Expansion in simple groups of Lie type}, in preparation.

\bibitem{carter} R.~W.~Carter, \emph{Simple groups of Lie type,} Reprint of the 1972 original. Wiley Classics Library. A Wiley-Interscience Publication. John Wiley \& Sons, Inc., New York, 1989. x+335 pp.

\bibitem{dickson} L.E. Dickson, \emph{Linear groups with an exposition of Galois Field Theory}, 2007 reprinting of the 1901 edition, Cosimo classics, New York.
\bibitem{dinai} O. Dinai, \emph{Expansion properties of finite simple groups}, Hebrew University Ph.D. thesis 2009, available at \texttt{arxiv:math/1001.5069}.
\bibitem{ghssv} A.~Gamburd, S.~Hoory, M.~Shahshahani, A.~Shalev and B.~Vir\'ag, \emph{On the girth of random Cayley graphs,} Random Structures Algorithms \textbf{35} (2009), no. 1, 100--117.
\bibitem{gowers} W.~T.~Gowers, \emph{Quasirandom groups,} Combin. Probab. Comput. \textbf{17} (2008), no. 3, 363--387.
\bibitem{green-survey} B.~J.~Green, \emph{Approximate groups and their applications: work of Bourgain, Gamburd, Helfgott and Sarnak,} Current Events Bulletin of the AMS, 2010.
\bibitem{hadad} U.~Hadad, \emph{On the shortest identity in finite simple groups of Lie type,} to appear in J. Group theory, available at \texttt{arXiv:math/0808.0622}.
\bibitem{helfgott} H.~A~Helfgott, \emph{Growth and generation in ${\rm SL}_2(\Z/p\Z)$}, Ann. of Math. (2) 167 (2008), no. 2, 601--623.


\bibitem{hoory-linial-wigderson} S.~Hoory, N.~Linial and A.~Wigderson, \emph{Expander graphs and their applications,}  Bull. Amer. Math. Soc. (N.S.) \textbf{43} (2006), no. 4, 439--561.
\bibitem{hrushovski} E.~Hrushovski, \emph{The elementary theory of the Frobenius automorphisms,} available at
\texttt{arXiv:math/0406514}.
\bibitem{jones} G.A. Jones, \emph{Varieties and simple groups}, J. Austr. Math. Soc. 17 (1974), 163-173.
\bibitem{kesten} H. Kesten, \emph{Symmetric random walks on groups}, Trans. Amer. Math. Soc. \textbf{92}, (1959), p. 336-354.
\bibitem{pnas} M.~Kassabov, A.~Lubotzky and N.~Nikolov, \emph{Finite simple groups as expanders,} Proc. Natl. Acad. Sci. USA \textbf{103} (2006), no. 16, 6116--6119.
\bibitem{landazuri-seitz} V.~Landazuri and G.~M.~Seitz, \emph{On the minimal degrees of projective representations of the finite Chevalley groups,} J. Algebra \textbf{32} (1974), 418--443.
\bibitem{larsen} M.~Larsen, \emph{Word maps have large image,} Israel J. Math. \textbf{139} (2004), 149--156.
\bibitem{lubotzky}  A.~Lubotzky, \emph{Finite simple groups of Lie type as expanders,} \texttt{arxiv:math/0904.3411}.
\bibitem{pyber-szabo} L.~Pyber and E.~Szabo, \emph{Growth in finite simple groups of Lie type,} announcement.
\bibitem{steinberg} R.~Steinberg, \emph{Generators for simple groups,} Canad. J. Math. \textbf{14} (1962), 277--283.
\bibitem{suzuki} M.~Suzuki, \emph{A new type of simple groups of finite order,} A new type of simple groups of finite order. Proc. Nat. Acad. Sci. U.S.A. \textbf{46} (1960), 868--870.

\bibitem{wilson} R.~A.~Wilson, \emph{The finite simple groups,} Graduate Texts in Mathematics \textbf{251}. Springer-Verlag London, Ltd., London, 2009. xvi+298 pp.



\end{thebibliography}
\end{document}